\documentclass[11pt]{article}

\usepackage[utf8]{inputenc}
\usepackage{amsmath, amssymb, amsthm, amsfonts}
\usepackage[margin=1in]{geometry}
\usepackage{enumitem}
\usepackage[hidelinks]{hyperref}
\usepackage{booktabs}

\newtheorem{theorem}{Theorem}[section]
\newtheorem{lemma}[theorem]{Lemma}
\newtheorem{proposition}[theorem]{Proposition}
\newtheorem{corollary}[theorem]{Corollary}
\newtheorem{conjecture}[theorem]{Conjecture}

\theoremstyle{definition}
\newtheorem{definition}[theorem]{Definition}
\newtheorem{remark}[theorem]{Remark}

\newtheorem{example}[theorem]{Example}

\newcommand{\armwalk}{\mathrm{arm\_walk}}

\title{Alternating Extremes in Graceful Labelings of \\
       Full Binary Trees and Spider Trees}
\author{%
  Bogdan Dumitru\thanks{Faculty of Mathematics and Computer Science,
  University of Bucharest, Romania.
  Emails: \texttt{bogdan.dumitru@fmi.unibuc.ro},
  \texttt{mihainacu906@gmail.com}.}
  \and
  Mihai Nacu\footnotemark[1]%
}
\date{}

\begin{document}
\maketitle

\begin{abstract}
We study two related uses of alternating extreme walks in graceful labelings.
For full binary trees, we formulate the pinned-spine conjecture: some deepest
root-to-leaf path should be labelable by the alternating extreme pattern
$0,n-1,1,n-2,\dots$ in a graceful labeling.  We prove the conjecture for comb
full binary trees.  We also prove a segregation lemma showing that such a spine
forces all off-spine labels into the middle interval and all off-spine
differences into the smaller differences.  Exhaustive search verifies the
pinned-spine conjecture for rooted non-isomorphic full binary trees through
order $23$, while a recursive heuristic succeeds except for one tree of order
$13$; a one-leaf relaxation succeeds through order $25$.

For spider trees, we use alternating walks on arms.  We prove a self-matched
packing theorem: pairwise disjoint self-matched legs based at hub label $1$, at
least one of which contains label $0$, can be combined into a graceful spider,
with unused labels attached as leaves.  This gives graceful labelings for
spiders $S(c,L_1,\dots,L_r,1^t)$ for all sufficiently large $t$.  We also record
a safe-zone reduction for the longest arm and formulate the remaining six-arm
problem as an offset five-arm residual problem.
\end{abstract}

\begin{quotation}
\end{quotation}

\section{Preliminaries}
\label{sec:prelim}

Let $T$ be a tree with vertex set $V(T)$ and edge set $E(T)$, with
$|V(T)|=n$ and $|E(T)|=n-1$.  A \emph{graceful labeling} of $T$ is a bijection
\[
   f:V(T)\longrightarrow \{0,1,\dots,n-1\}
\]
such that the induced edge labels
\[
   \{|f(u)-f(v)|:uv\in E(T)\}
\]
are exactly $\{1,2,\dots,n-1\}$.  The Graceful Tree Conjecture of Ringel and
Kotzig, often attributed to Rosa~\cite{Rosa1967}, asserts that every tree is
graceful.  The conjecture remains open; Aldred and McKay verified it through
order~$27$~\cite{AldredMcKay1998}, and Fang later extended the computation
through order~$35$~\cite{Fang2010}.

A \emph{full binary tree} (FBT) is a rooted tree in which every non-leaf vertex
has exactly two children.  The root is counted as internal unless the tree
consists of the root alone.  If an FBT has $m$ internal vertices, then it has
$m+1$ leaves and $n=2m+1$ vertices.

A \emph{spider tree} (ST) $S(\ell_1,\dots,\ell_k)$ is obtained from a central
hub vertex by attaching $k$ paths, called arms, at the hub.  The integer
$\ell_i$ is the number of edges in the $i$-th arm.  Thus the ST has
$1+\sum_i \ell_i$ vertices.

Throughout the FBT section, path length means number of edges.  Thus a chosen
root-to-leaf spine of length $\ell$ has vertices $v_0,\dots,v_\ell$.  In the ST
part of the paper the same convention is used for arm lengths: the parameter
$\ell_i$ counts edges, not vertices.

We shall use the following notation for alternating walks along a path.
Starting from a label
$c\in[0,n)$, choose an ordered list of differences
$D=(d_0,d_1,\dots,d_{m-1})$ and an initial sign
$\varepsilon\in\{+1,-1\}$.  Define
\[
   x_0=c,\qquad
   x_{i+1}=x_i+(-1)^i\varepsilon d_i
   \quad (i=0,\dots,m-1).
\]
We write the resulting sequence of new labels as
$\armwalk(c,D,\varepsilon,n)=(x_1,\dots,x_m)$.  The walk is valid if all labels
lie in $[0,n)$ and are distinct from the labels already in use.  By
construction, its edge differences are exactly the entries of $D$ in order.

The basic example is the alternating extreme sequence
\begin{equation}
\label{eq:alt-extreme}
   0,\ n-1,\ 1,\ n-2,\ 2,\ n-3,\ \dots ,
\end{equation}
obtained from $c=0$, $D=(n-1,n-2,\dots,1)$, and initial sign $+1$.  This is
Rosa's canonical graceful labeling of a path~\cite{Rosa1967}.

\section{Full Binary Trees}
\label{sec:fbt}

For FBTs we study a pinned version of graceful labeling.  Instead of only
asking for a graceful labeling, we ask whether some deepest root-to-leaf path
can carry the alternating extreme sequence
\eqref{eq:alt-extreme}.  This is stronger than ordinary gracefulness: the
spine consumes the largest edge differences, leaving the side branches with a
constrained residual problem.  We do not settle plain gracefulness of FBTs,
which is a special case of the Graceful Tree Conjecture; we study the stronger
pinned form and what it forces.

There is earlier work on binary and related trees.  Cahit asked whether all
complete binary trees are graceful~\cite{Cahit1976}; this complete-tree case is
covered by earlier balanced-tree constructions of Stanton and
Zarnke~\cite{StantonZarnke1973}, and by the later interlaced-tree framework of
Koh, Rogers, and Tan~\cite{KohRogersTan1979,KohRogersTan1981}.  Unlike those
constructions, the present question allows an arbitrary FBT and requires a
deepest root-to-leaf path to carry the alternating extreme pattern.  The
off-spine vertices and edges must then use the remaining labels and differences,
respectively.

\subsection{The pinned-spine conjecture}

\begin{definition}[Alternating extreme spine]
\label{def:alt-spine}
Let $T$ be an FBT on $n$ vertices, and let
\[
   v_0v_1\cdots v_\ell
\]
be a longest root-to-leaf path, with $v_0$ the root.  The alternating extreme
spine assignment is
\[
   s(v_i)=
   \begin{cases}
      i/2, & i\text{ even},\\[2pt]
      n-1-(i-1)/2, & i\text{ odd}.
   \end{cases}
\]
Thus the spine labels are $0,n-1,1,n-2,\dots$, truncated after $\ell+1$
terms.
\end{definition}

\begin{conjecture}[Pinned-spine conjecture for FBTs]
\label{conj:pinned-spine}
Every FBT $T$ on $n$ vertices admits a graceful labeling $f$ such that, for
some longest root-to-leaf path $v_0v_1\cdots v_\ell$, the restriction of $f$ to
that path is the alternating extreme spine assignment.
\end{conjecture}

Once such a longest path is chosen, the remaining vertices are labeled under
the restrictions imposed by that choice.  We next state those restrictions, and
then prove the conjecture for the comb family.

\subsection{Forced segregation}

\begin{lemma}[Spine segregation]
\label{lem:segregation}
Suppose $f$ is a graceful labeling of an FBT $T$ whose longest root-to-leaf
path $v_0v_1\cdots v_\ell$ carries the alternating extreme spine assignment.
Let
\[
   a=\left\lceil\frac{\ell+1}{2}\right\rceil,
   \qquad
   b=\left\lfloor\frac{\ell+1}{2}\right\rfloor .
\]
Then the spine uses the labels
\[
   \{0,1,\dots,a-1\}\cup\{n-b,n-b+1,\dots,n-1\},
\]
and the spine edges use exactly the differences
\[
   \{n-\ell,n-\ell+1,\dots,n-1\}.
\]
Consequently, all off-spine vertices are labeled from the middle block
\[
   M=\{a,a+1,\dots,n-1-b\},
\]
and the off-spine edges must realize exactly
\[
   \{1,2,\dots,n-\ell-1\}.
\]
\end{lemma}

\begin{proof}
The even spine positions receive the labels $0,1,\dots,a-1$, while the odd
spine positions receive $n-1,n-2,\dots,n-b$.  Consecutive spine labels differ
by
\[
   |s(v_{i+1})-s(v_i)|=n-1-i
   \qquad (i=0,\dots,\ell-1),
\]
so the spine differences are $n-1,n-2,\dots,n-\ell$.  A graceful labeling uses
each label and each nonzero difference exactly once, so the remaining labels
and differences are precisely the stated middle block and lower difference
block.
\end{proof}

Once the spine is fixed, the remaining labels and differences are determined as
sets: the side branches must live in the middle labels and must spend the small
differences.

\subsection{The comb FBT}

We next treat the comb FBT.  For $\ell\ge1$, let $C_\ell$ be the rooted full
binary tree with internal vertices
\[
   v_0,v_1,\dots,v_{\ell-1}
\]
and leaves
\[
   w_0,w_1,\dots,w_{\ell-1},v_\ell .
\]
The parent--children relation is
\[
   v_i\to\{v_{i+1},w_i\}\qquad(0\le i\le \ell-1).
\]
Thus $v_0v_1\cdots v_\ell$ is a root-to-leaf path, and each internal vertex on
this path has its other child equal to a leaf.  The tree has $2\ell+1$ vertices.
This is the ``comb'' FBT: the path $v_0v_1\cdots v_\ell$ is the spine, and the
leaves $w_i$ are the teeth.

As an unrooted tree, $C_\ell$ is a caterpillar, so Rosa's theorem already
establishes its gracefulness~\cite{Rosa1967}.  Here we require the stronger
pinned-spine condition.

\begin{theorem}[Comb FBT]
\label{thm:comb-fbt}
For every $\ell\ge1$, the comb FBT $C_\ell$ admits a graceful labeling
satisfying the pinned-spine conjecture.
\end{theorem}

\begin{proof}
Put the alternating extreme labels on the spine:
\[
   f(v_i)=
   \begin{cases}
      i/2, & i\text{ even},\\[2pt]
      n-1-(i-1)/2, & i\text{ odd}.
   \end{cases}
\]
The remaining labels form the middle block of Lemma~\ref{lem:segregation}.
We assign them to the leaves so that the leaf-edge differences are exactly
$1,2,\dots,\ell$.

Let $h=\ell+1$ be the number of spine vertices.  First suppose $h=2k$.  Then
$n=4k-1$ and $\ell=2k-1$.  The middle block is
\[
   M=\{k,k+1,\dots,3k-2\}.
\]
Define
\[
   f(w_{2j})=2k-1-j\quad (0\le j\le k-1),
   \qquad
   f(w_{2j+1})=2k+j\quad (0\le j\le k-2).
\]
For even spine positions,
\[
   |f(w_{2j})-f(v_{2j})|
      =(2k-1-j)-j
      =2k-1-2j,
\]
which gives the odd differences $2k-1,2k-3,\dots,1$.  For odd spine positions,
\[
   |f(w_{2j+1})-f(v_{2j+1})|
      =(4k-2-j)-(2k+j)
      =2k-2-2j,
\]
which gives the even differences $2k-2,2k-4,\dots,2$.  Together these are
$\{1,\dots,2k-1\}=\{1,\dots,\ell\}$.

Now suppose $h=2k+1$.  Then $n=4k+1$ and $\ell=2k$.  The middle block is
\[
   M=\{k+1,k+2,\dots,3k\}.
\]
Define
\[
   f(w_{2j})=2k-j\quad (0\le j\le k-1),
   \qquad
   f(w_{2j+1})=2k+1+j\quad (0\le j\le k-1).
\]
For even spine positions,
\[
   |f(w_{2j})-f(v_{2j})|
      =(2k-j)-j
      =2k-2j,
\]
which gives the even differences $2k,2k-2,\dots,2$.  For odd spine positions,
\[
   |f(w_{2j+1})-f(v_{2j+1})|
      =(4k-j)-(2k+1+j)
      =2k-1-2j,
\]
which gives the odd differences $2k-1,2k-3,\dots,1$.  Together these are
$\{1,\dots,2k\}=\{1,\dots,\ell\}$.

In both cases the spine differences are
$\{\ell+1,\ell+2,\dots,n-1\}$ by Lemma~\ref{lem:segregation}, while the leaf
edges produce $\{1,\dots,\ell\}$.  The labels are a bijection onto
$\{0,\dots,n-1\}$, so the labeling is graceful.
\end{proof}

The proof exhibits a useful monotonic pairing.  Leaves attached to low spine
labels receive the lower part of the middle block in decreasing order; leaves
attached to high spine labels receive the upper part in increasing order.  The
two monotone lists produce opposite parities of differences and interleave to
fill $\{1,\dots,\ell\}$.  In this family, the pinned spine determines a direct
assignment of all remaining labels.

\subsection{What happens off the spine}

For general FBTs, side branches are not single leaves.  Lemma~\ref{lem:segregation}
still says that all off-spine labels lie in the middle block $M$, but individual
side subtrees need not receive intervals inside $M$.  In our computations,
gapped label sets occur on side subtrees, and the internal
differences assigned to a side subtree are also gapped because the connecting
edge to the spine spends one of the small differences.

In the examples found by search, a side subtree is often labeled as if it were
using alternating extremes of its own allocated label set, even when that set
has gaps:
\[
   a_1,\ a_k,\ a_2,\ a_{k-1},\dots
   \qquad
   (a_1<\cdots<a_k).
\]
This observation is not a proof, but it explains why a recursive spine-first
strategy is reasonable.

\subsection{Search and verification}

The verification in Table~\ref{tab:fbt-empirical} uses an exact pinned-spine
search.  For each rooted non-isomorphic FBT, the program tries each longest
root-to-leaf path, fixes the alternating extreme labels on that path, and then
searches for a completion on the off-spine vertices using the remaining labels
and differences from Lemma~\ref{lem:segregation}.  Thus the last column of the
table is a direct computational test of the pinned-spine conjecture in the stated
range.

The table also records the strict recursive heuristic.  After the main spine is
fixed, each side subtree is treated in the same way: choose a longest local
spine, label it by the extremes of the label set assigned to that subtree, and
continue recursively.  The convention is as follows.  If a side subtree with assigned label set
$S=\{a_1<\cdots<a_r\}$ hangs from a parent with a high label, then the local
spine starts at $a_1$; if it hangs from a parent with a low label, then the
local spine starts at $a_r$.  Thus the first edge into the side subtree crosses
from high to low or from low to high.

The strict heuristic is useful but not complete.  It fails on one $13$-vertex FBT,
although the exact pinned-spine search still finds a labeling there.  All rooted
non-isomorphic FBTs through order $23$ were generated recursively by splitting
the two child subtrees at the root and identifying mirror duplicates by
canonical rooted signatures.

\begin{table}[ht]
\centering
\begin{tabular}{rrrr}
\toprule
$n$ & FBTs tested & strict heuristic & pinned-spine labeling found \\
\midrule
5  & 1   & 1/1     & 1/1 \\
7  & 2   & 2/2     & 2/2 \\
9  & 3   & 3/3     & 3/3 \\
11 & 6   & 6/6     & 6/6 \\
13 & 11  & 10/11   & 11/11 \\
15 & 23  & 23/23   & 23/23 \\
17 & 46  & 46/46   & 46/46 \\
19 & 98  & 98/98   & 98/98 \\
21 & 207 & 207/207 & 207/207 \\
23 & 451 & 451/451 & 451/451 \\
\bottomrule
\end{tabular}
\caption{Exhaustive FBT data for rooted non-isomorphic FBTs through order
$23$.  The last column was checked by a search fixing the alternating spine.}
\label{tab:fbt-empirical}
\end{table}

A relaxed version of the heuristic allows only the terminal leaf of a local
spine to swap with another label assigned to the same local subtree.  This
one-leaf relaxation removes the single strict-heuristic failure and succeeds on
every rooted non-isomorphic FBT tested through order $25$.

\subsection{A threshold obstruction}

The search data raise a stronger question: can every FBT admit a pinned-spine
$\alpha$-labeling?  An $\alpha$-labeling is a graceful labeling for which there
is an integer $\alpha$ such that every edge has one endpoint labeled at most
$\alpha$ and the other endpoint labeled greater than $\alpha$.  The following
example answers this question negatively.

\begin{example}[Threshold obstruction]
\label{prop:threshold-obstruction}
Let $T^*$ be the FBT with parent--children relation
\[
   0\to\{1,2\},\quad
   2\to\{3,10\},\quad
   3\to\{4,5\},\quad
   5\to\{6,7\},\quad
   7\to\{8,9\},\quad
   10\to\{11,12\}.
\]
Then $T^*$ admits a graceful labeling realizing the pinned-spine conjecture, but
no $\alpha$-labeling realizes the pinned-spine conjecture on $T^*$.
\end{example}

\begin{proof}
The labeling
\[
   (f(0),f(1),\dots,f(12))
   =(0,4,12,1,8,11,9,2,10,7,6,3,5)
\]
is graceful: its labels are $\{0,\dots,12\}$ and its edge differences are
$\{1,\dots,12\}$.  On the longest path
$0\to2\to3\to5\to7\to8$, the labels are
$0,12,1,11,2,10$, which is the alternating extreme spine.

For the obstruction, use the depth-parity bipartition.  The even-depth class is
$V_0=\{0,3,6,7,10\}$ and the odd-depth class is
$V_1=\{1,2,4,5,8,9,11,12\}$.  Any $\alpha$-labeling must put one bipartition
class below the threshold and the other above it.  Since the pinned spine
contains $0,1,2$ in $V_0$ and $12,11,10$ in $V_1$, it forces
$V_0$ to use $\{0,1,2,3,4\}$ and $V_1$ to use $\{5,\dots,12\}$.  The remaining
vertices $6$ and $10$ in $V_0$ must therefore receive labels $3$ and $4$.
But vertex $10$ is adjacent to vertex $2$, whose forced label is $12$, so the
edge $(2,10)$ has difference $8$ or $9$.  Both differences are already used on
the pinned spine.  This contradiction rules out a pinned-spine
$\alpha$-labeling.
\end{proof}

Thus a proof of the pinned-spine conjecture cannot assume that every edge crosses a
fixed low-high threshold.

\subsection{FBT questions}

The FBT part leaves three concrete problems.
\begin{enumerate}[label=Q\arabic*.,itemsep=2pt]
   \item Prove the pinned-spine conjecture for all FBTs.
   \item Explain the label-set allocation on side subtrees, especially the
   gapped sets that appear in successful pinned-spine labelings.
   \item Extend the comb theorem from leaf side branches to FBTs in which some
   off-spine children root larger full binary subtrees.
\end{enumerate}

\section{Spider Trees}
\label{sec:spider}

We now turn to spider trees.  We write
$S(\ell_1,\dots,\ell_k)$ for an ST with arms sorted as
$\ell_1\le\cdots\le\ell_k$, and put
\[
   n=1+\sum_{i=1}^k \ell_i .
\]
Known results cover many important cases.  Huang, Kotzig, and Rosa proved that
STs with three or four arms are graceful~\cite{HuangKotzigRosa1982}.  More
recently, STs with at most five arms were shown to be
graceful~\cite{Panpa2025}; STs with at most four arms of length greater than
one are graceful~\cite{Hindawi4legs}; and STs whose arm lengths differ by at
most one are graceful~\cite{BahlsLakeWertheim}.  Kanetkar and Sane also treat
the arithmetic-progression quasistars with path lengths
$1,d+1,\dots,(k-1)d+1$ for $k\le6$~\cite{KanetkarSane2007}.  In concurrent work,
Shan and Zhong treat two further families: arms with fast-growing lengths, and
spiders with one long arm and all other arms of length at most
two~\cite{ShanZhong2026}.  These constructions address a different regime from
the cited results.  We use a safe-zone reduction after fixing the longest arm
and a hub-one packing construction in which unused labels are attached as
leaves.  We then formulate the six-arm residual problem, to which the known
five-arm theorem does not directly apply.

\subsection{The safe zone}

The longest arm plays the same role as the pinned spine in the FBT section: it
spends the largest differences and leaves a controlled residual problem.  This
is the standard zigzag setup; we restate it in the form we use later.

\begin{lemma}[Safe zone]
\label{lem:st-safe-zone}
Let $S(\ell_1,\dots,\ell_k)$ be an ST with $k\ge3$ and
$\ell_1\le\cdots\le\ell_k$.  Set $R=n-1-\ell_k$.  Define
\[
H=
\begin{cases}
\ell_k/2, & \ell_k \text{ even},\\[2pt]
n-(\ell_k+1)/2, & \ell_k \text{ odd}.
\end{cases}
\]
If the hub receives label $H$ and the primary arm of length $\ell_k$ is labeled
by
\[
   \armwalk\bigl(H,(n-\ell_k,n-\ell_k+1,\dots,n-1),\varepsilon,n\bigr),
\]
with $\varepsilon=+1$ for even $\ell_k$ and $\varepsilon=-1$ for odd
$\ell_k$, then that arm uses exactly the differences
$\{n-\ell_k,\dots,n-1\}$.  The unused labels form one interval $[A,B]$ of size
$R$, adjacent to the hub: $H=A-1$ when $\ell_k$ is even and $H=B+1$ when
$\ell_k$ is odd.  The remaining differences are $\{1,\dots,R\}$.
\end{lemma}

\begin{proof}
Write $\ell_k=2m$ first.  Then $H=m$.  The walk alternates
\[
   m,\ n-m,\ m-1,\ n-m+1,\ \dots,\ 0,\ n-1.
\]
Thus the used labels are
\[
   \{0,\dots,m\}\cup\{n-m,\dots,n-1\},
\]
and the unused labels are $[m+1,n-m-1]$, an interval of size
$n-1-2m=R$.  The consumed differences are the entries of the prescribed
difference list.

If $\ell_k=2m+1$, then $H=n-m-1$ and the same computation, reflected about the
middle of the label set, gives used labels
\[
   \{0,\dots,m\}\cup\{n-m-1,\dots,n-1\}.
\]
The unused interval is $[m+1,n-m-2]$, again of size $R$, and the hub is one
unit above it.
\end{proof}

We use this reduction mainly for search and for the six-arm
residual problem below.

\subsection{Hub-one self-matched packing}

We next use a different viewpoint.  Put the hub at label $1$ and write
$m=n-1$.  The non-hub labels are
\[
   A_n=\{0,2,3,\dots,m\}.
\]
Define a matching map $\tau_n:A_n\to\{1,\dots,m\}$ by
\[
   \tau_n(0)=m,\qquad \tau_n(x)=x-1\quad(x\ge2).
\]
For $x\ge2$, this is exactly the difference made by attaching $x$ directly to
the hub as a leaf.

\begin{definition}[Self-matched leg]
\label{def:self-matched-leg}
A rooted label path
\[
   P=(1,x_1,x_2,\dots,x_L)
\]
with distinct $x_i\in A_n$ is a self-matched leg if
\[
   \{|1-x_1|,|x_1-x_2|,\dots,|x_{L-1}-x_L|\}
   =
   \tau_n(\{x_1,\dots,x_L\}).
\]
In the packing theorem below we will need at least one self-matched leg to
contain the label $0$.
\end{definition}

\begin{theorem}[Self-matched packing]
\label{thm:self-matched-packing}
Let $P_1,\dots,P_r$ be self-matched legs, all rooted at the common hub label
$1$.  Suppose their non-hub label sets are pairwise disjoint and that at least
one of them contains $0$.  Attach these paths as the nontrivial arms of an ST,
and attach every unused non-hub label as a leaf at the hub.  The resulting
labeling is graceful.
\end{theorem}

\begin{proof}
Let $X$ be the union of the labels used on the self-matched legs.  Since
$\tau_n$ is injective and each leg is self-matched, the nontrivial arms use
exactly the difference set $\tau_n(X)$.  Let $Y=A_n\setminus X$.  Since
$0\in X$, every label in $Y$ is at least $2$, and a leaf labeled $y\in Y$
contributes the difference $y-1=\tau_n(y)$.  Thus the leaves use
$\tau_n(Y)$.  Together the arms and leaves use
\[
   \tau_n(X)\sqcup\tau_n(Y)=\tau_n(A_n)=\{1,\dots,m\},
\]
and the vertex labels are exactly $\{0,1,\dots,m\}$.
\end{proof}

\subsection{Explicit self-matched legs}

We now give two explicit sources of self-matched legs.  The first supplies
legs of arbitrary prescribed length, not necessarily containing label $0$; the
second supplies one leg containing label $0$, which is needed to close the
packing theorem.

\begin{theorem}[Multiplicative legs]
\label{thm:multiplicative-legs}
Let $L\ge2$ and $q\ge1$, with $n\ge Lq+2$.  Define
$P_L(q)=(1,x_1,\dots,x_L)$ by
\[
   x_{2j+1}=(L-j)q+1\quad(j\ge0),\qquad
   x_{2j}=jq+1\quad(j\ge1),
\]
truncated after $L$ non-hub labels.  Then $P_L(q)$ is a self-matched leg with
label set
\[
   \{q+1,2q+1,\dots,Lq+1\}
\]
and difference set
\[
   \{q,2q,\dots,Lq\}.
\]
\end{theorem}

\begin{proof}
The path lists the arithmetic progression
$q+1,2q+1,\dots,Lq+1$ in alternating high-low order.  The edge from the hub has
difference $Lq$, and the remaining edges give $(L-1)q,\dots,q$.  Since each
label is at least $2$, applying $\tau_n$ to the label set gives the same
difference set.
\end{proof}

The packing theorem needs at least one self-matched leg containing label $0$.
The following templates provide such a leg in every length.

\begin{theorem}[Closure templates]
\label{thm:closure-templates}
For every $L\ge2$ and every allowed $m$ below, let $n=m+1$ and define
\[
C_L(m)=
\begin{cases}
(1,m,0), & L=2,\quad m\ge2,\\[2pt]
(1,m-1,0,m), & L=3,\quad m\ge3,\\[2pt]
(1,2^{L-4}+1,2^{L-5}+1,\dots,2^0+1,m,0,m-1),
   & L\ge4,\quad m\ge2^{L-4}+3 .
\end{cases}
\]
Then $C_L(m)$ is a self-matched leg of length $L$ containing $0$ in the
ambient label set $\{0,\dots,m\}$.
\end{theorem}

\begin{proof}
The cases $L=2$ and $L=3$ are immediate:
\[
   (1,m,0) \quad\text{has differences}\quad \{m-1,m\},
\]
and
\[
   (1,m-1,0,m) \quad\text{has differences}\quad \{m-2,m-1,m\}.
\]
These are exactly the $\tau_n$-images of their non-hub label sets.

Now let $L\ge4$ and write $L=r+3$.  The prefix
\[
   1,\ 2^{r-1}+1,\ 2^{r-2}+1,\ \dots,\ 3,\ 2
\]
has consecutive differences
\[
   2^{r-1},2^{r-2},\dots,1.
\]
The final three edges contribute
\[
   |2-m|=m-2,\qquad |m-0|=m,\qquad |0-(m-1)|=m-1.
\]
Thus the edge differences are
\[
   \{1,2,4,\dots,2^{r-1}\}\cup\{m-2,m-1,m\}.
\]
The map $\tau_n$ sends the low labels $2^{j-1}+1$ to the powers
$2^{j-1}$, and sends $m,0,m-1$ to $m-1,m,m-2$.  The lower bound on $m$
ensures all labels are distinct.
\end{proof}

\begin{corollary}[Leaf-extended STs]
\label{cor:leaf-extended-st}
Fix $c\ge2$ and nontrivial arm lengths $L_1,\dots,L_r\ge2$.  For all
sufficiently large $t$, the ST
\[
   S(c,L_1,\dots,L_r,1^t)
\]
has a graceful labeling with hub label $1$.
\end{corollary}

\begin{proof}
Let $M_c=2$ for $c=2$, $M_c=3$ for $c=3$, and
$M_c=2^{c-4}+3$ for $c\ge4$.  Choose an integer
\[
   B>\max\{M_c,L_1,\dots,L_r\}
\]
and set $q_i=B^i$.  Use the closure template $C_c(m)$ on the arm of length
$c$, and use the multiplicative leg $P_{L_i}(q_i)$ on the arm of length
$L_i$.  The label sets of the multiplicative legs are pairwise disjoint
because equality $j_1B^{i_1}=j_2B^{i_2}$ with $i_1<i_2$ would force
$j_1\ge B$, contradicting $j_1\le L_{i_1}<B$.  For large enough $t$, the
global label maximum $m=n-1$ satisfies the bound needed for $C_c(m)$, and
$m-1$ is larger than every label used by the multiplicative legs.  Hence the
chosen legs are pairwise label-disjoint and one contains $0$.  The conclusion
follows from Theorem~\ref{thm:self-matched-packing}.
\end{proof}

\begin{remark}[Relation to concurrent work]
\label{rem:concurrent}
The construction in Corollary~\ref{cor:leaf-extended-st} first packs several
arms of prescribed lengths and then attaches all unused labels as hub leaves.
The concurrent families of Shan and Zhong~\cite{ShanZhong2026} have a different
shape: one requires fast-growing arm lengths, and the other allows only one
long arm.  These results do not cover the padded mixed-length regime treated by
Corollary~\ref{cor:leaf-extended-st}.  The corollary requires $t$ to be
sufficiently large and does not address fixed arm-length profiles with few
leaves.
\end{remark}

\subsection{Largest-difference-first}

The previous corollary handles STs with enough leaves.  For general STs we also
use a largest-difference-first search.  After the primary arm is placed by
Lemma~\ref{lem:st-safe-zone}, the remaining arms are processed in decreasing
order of length.  Along the current arm, the search tries unused differences in
decreasing order; for each difference it tries the two possible signs and keeps
the first move landing on an unused label.  If a later step fails, the search
backtracks to the most recent choice.

LDF is an ordering within a depth-first search rather than a deterministic
labeling rule.  It found a labeling for every tested instance listed below, and
a separate graceful verifier checked every output.  These computations do not
prove that the ordering always succeeds.

\begin{center}
\begin{tabular}{rl}
\toprule
arms & tested range \\
\midrule
$3$ & all $S(\ell_a,\ell_b,\ell_p)$ with $\ell_p\le55$ \\
$4$ & all $S(\ell_1,\dots,\ell_4)$ with $\max_i\ell_i\le30$ \\
$5$ & all $S(\ell_1,\dots,\ell_5)$ with $\max_i\ell_i\le20$ \\
$6$ & all configurations with $\sum_i\ell_i\le34$ \\
$7$--$10$ & all configurations with $\sum_i\ell_i\le22$ for each arm count \\
\bottomrule
\end{tabular}
\end{center}

\subsection{A six-arm reduction target}
\label{subsec:offset-five}

Spiders with six arms are the first case not covered by the results cited above.
Applying the safe-zone construction to a longest arm uses the largest
differences and leaves a five-arm residual problem.  This residual problem is
not ordinary five-arm gracefulness because the hub position is inherited from
the first arm.

\begin{proposition}[Hub-zero five-arm reduction]
\label{prop:hub-zero-five-arm-reduction}
If every five-arm spider admits a graceful labeling with hub label $0$, then
every six-arm spider is graceful.
\end{proposition}

\begin{proof}
Let the six-arm spider have total edge count $M$, and choose a longest arm of length
$L$.  Apply Lemma~\ref{lem:st-safe-zone} to this arm.  The chosen arm uses the
differences $M-L+1,\dots,M$, while the unused labels form an interval
$[A,B]$ of size $s=M-L$ adjacent to the hub label $H$.  The remaining
differences are $1,\dots,s$.

The other five arms form a five-arm spider with $s$ edges.  By hypothesis,
this spider has a graceful labeling with hub label $0$ and non-hub labels
$1,\dots,s$.  If $H=A-1$, add $H$ to every non-hub label of that labeling.  If
$H=B+1$, replace every non-hub label $x$ by $H-x$.  In both cases the residual
labels become exactly $[A,B]$, the hub remains $H$, and all residual edge
differences are preserved.  Together with the first arm, this gives a graceful
labeling of the original six-arm spider.
\end{proof}

The proposition uses a stronger input than ordinary five-arm gracefulness.  The
following offset formulation records the residual problem left by a fixed first
arm.

Let a six-arm spider have total edge count $M$, and suppose one chosen arm has
even length $2q$.  Write $s=M-2q$ for the total length of the other five arms.
Consider a labeling of the chosen arm which uses the outside labels
\[
   \{0,1,\dots,q\}\cup\{M-q+1,M-q+2,\dots,M\}
\]
and realizes the top differences
\[
   s+1,s+2,\dots,M.
\]
If this arm starts at the hub label $k\in\{0,\dots,q\}$, then the unused labels
are exactly
\[
   q+1,q+2,\dots,q+s,
\]
and the unused differences are $1,2,\dots,s$.

Let $S$ be a five-arm spider with $s$ edges.  For $0\le c\le s-1$, an
\emph{offset hub-zero labeling} of $S$ with offset $c$ is a labeling with hub
label $0$, non-hub labels
\[
   c+1,c+2,\dots,c+s,
\]
and edge differences exactly
\[
   1,2,\dots,s.
\]

With this terminology, the even-arm reduction is immediate.  If the chosen
six-arm path above starts at $k$ with $0\le q-k\le s-1$, and the remaining
five-arm spider has an offset hub-zero labeling with
\[
   c=q-k,
\]
then adding $k$ to every non-hub label in that five-arm labeling and using hub
label $k$ gives a graceful labeling of the original six-arm spider.  Conversely,
any graceful labeling built from such an outside-label first arm restricts to
an offset hub-zero labeling of the residual five-arm spider after subtracting
$k$ from the residual labels.  This statement is conditional only on the
displayed top-difference ordering of the first arm; the usual zigzag gives
$k=q$, hence $c=0$.

The odd case is the same after reflection.  If the chosen arm has length
$2q+1$, use outside labels
\[
   \{0,1,\dots,q\}\cup\{M-q,M-q+1,\dots,M\}.
\]
If the arm starts at hub label $M-q+c$ with
$0\le c\le\min\{q,s-1\}$, then reflecting the residual labels by
$y\mapsto M-q+c-y$ gives an offset hub-zero problem with offset $c$.

Thus the known five-arm theorem does not give exactly the input needed here.
For this reduction one needs a hub-pinned, shifted five-arm labeling.  The case
$c=0$ is the ordinary hub-zero five-arm problem.  Varying $c$ changes the
non-hub label interval while the required difference set remains
$\{1,\dots,s\}$.

The offset problem can also be normalized by subtracting $c$ from the non-hub
labels while leaving the hub label equal to $0$.  The five arms then cover
$\{1,\dots,s\}$.
If $a_i$ is the first normalized label on the $i$-th arm, then the corresponding
hub-edge difference in the offset labeling is $a_i+c$.  Hence the normalized
problem is to cover $\{1,\dots,s\}$ by five paths whose hub-edge differences
\[
   a_1+c,\dots,a_5+c
\]
together with their internal edge differences give $\{1,\dots,s\}$.
Since no internal edge between labels in $\{1,\dots,s\}$ has difference $s$,
one first label is forced to be $s-c$.

A stronger offset problem requires the five
hub-edge differences to be
\[
   s-4,s-3,s-2,s-1,s.
\]
Equivalently, in normalized labels the five arms must start from
\[
   s-c-4,\ s-c-3,\ s-c-2,\ s-c-1,\ s-c,
\]
and the internal edges must realize $1,2,\dots,s-5$.  This stronger version
requires $0\le c\le s-5$.

\subsection{ST questions}

The ST part leaves four concrete problems.
\begin{enumerate}[label=Q\arabic*.,itemsep=2pt]
   \item Prove an offset five-arm residual theorem, especially the strengthened
   form above, strong enough to close the six-arm reduction.
   \item Find a direct packing criterion for self-matched legs of prescribed
   lengths, without requiring many added leaves.
   \item Explain why the LDF-ordered search is so small in practice, or find a
   deterministic subrule that removes backtracking.
   \item Find a closed four-arm constructor for arms of mixed lengths, in the
   spirit of the leaf-extended family above.
\end{enumerate}

Both parts of the paper use an alternating extreme walk along a path.  For FBTs,
the path is a longest root-to-leaf spine; for spiders, it is the primary arm and
also appears in the self-matched packing construction.  The resulting residual
problems differ: FBTs leave side subtrees with gapped label sets, whereas spiders
leave several arms using one safe-zone interval.

\paragraph{Code availability.}
The C++ programs used for the computations in this paper---the exact
pinned-spine search behind Table~\ref{tab:fbt-empirical} together with its
one-leaf relaxation, and the LDF depth-first search over the spider
ranges of Section~\ref{sec:spider}---are available in the accompanying
\href{https://github.com/bogdan27182/graceful-paper}{github repository}.

\bibliographystyle{plain}
\bibliography{references}

\end{document}